\documentclass[11pt]{amsart}
\usepackage{amssymb,amsthm,enumerate}
\usepackage{eucal,mathrsfs}
\usepackage[bbgreekl]{mathbbol}
\usepackage{color}
\usepackage[all]{xy}
\usepackage{pxfonts}

\theoremstyle{plain}
\newtheorem{lem}{Lemma}[section]

\newtheorem{thm}[lem]{Theorem}
\newtheorem*{thm*}{Theorem}

\newtheorem{prop}[lem]{Proposition}

\theoremstyle{definition}
\newtheorem{defn}[lem]{Definition}
\newtheorem{rem}[lem]{Remark}
\newtheorem{ex}[lem]{Example}

\newcommand{\mbb}[1]{\mathbb #1}

\newcommand{\oper}[1]{\operatorname{#1}}

\usepackage[OT2,T1]{fontenc}
\DeclareSymbolFont{cyrletters}{OT2}{wncyr}{m}{n}
\DeclareMathSymbol{\Sha}{\mathalpha}{cyrletters}{"58}

\newcommand{\Br}{\oper{Br}}

\newcommand{\lb}{\left<\right.\!}
\newcommand{\rb}{\!\left.\right>}

\newcommand{\lbb}{\lb\!\lb}
\newcommand{\rbb}{\rb\!\rb}

\newcommand{\ddim}{\oper{ddim}}
\newcommand{\cdim}{\oper{cd}}

\newcommand{\N}{\oper{N}}
\newcommand{\cof}{\oper{adj}}

\title{Diophantine and cohomological dimensions}
\author{Daniel Krashen}
\author{Eliyahu Matzri}

\thanks{The first author was partially supported by NSF grants DMS-1007462 and DMS-1151252}
\thanks{The second author was partially supported by the Kreitman foundation}
\date{}

\begin{document}
\begin{abstract}
We give explicit linear bounds on the $p$-cohomological
dimension of a field in terms of its Diophantine dimension. In
particular, we show that for a field of Diophantine dimension at most
$4$, the $3$-cohomological dimension is less than or equal to the
Diophantine dimension.
\end{abstract}

\maketitle


\section*{Introduction}

Let $F$ be a field and suppose that $p$ is a prime integer not
equal to the characteristic of $F$. Recall that a field $F$ is said to
have the $C_d$ property if every homogeneous polynomial of degree $n$
over $F$ in more than $n^d$ variables has a nontrivial zero. The
Diophantine dimension of $F$, denoted $\ddim(F)$ is defined to be the
least integer $d$ such that $F$ has the $C_d$ property. We define
the $p$-cohomological dimension of $F$, denoted $\cdim_p(F)$ to be the
least integer $d$ such that the Galois cohomology groups $H^d(L, \mbb
Z/p\mbb Z)$ vanish for every $L/F$ finite. One says that $F$ has
cohomological dimension $d$, written $\cdim(F) = d$ if $d$ is the
maximum value of $\cdim_p(F)$ taken over all $p$.

Although there are well known examples of Ax \cite{Ax} of fields of
cohomological dimension $1$ which do not have the $C_d$ property for
any $d$, it is quite possible that having a bound on the Diophantine
dimension will give one a bound on the $p$-cohomological dimension.
In \cite[page 99]{Serre:CG}, after outlining why the Milnor
Conjectures would imply $\cdim_2(F) \leq \ddim(F)$, Serre states ``Il est
probable que ce r\'esultat est \'egalement valable pour $p \neq 2$.'' On the
other hand, is was not known until present whether or not a field with
$\ddim(F)$ finite must have $\cdim_p(F)$ finite in general for $p \neq
2$.

In this paper we give explicit bounds on the $p$-cohomological
dimensions in terms of the Diophantine dimension, and show in
particular that $\cdim_p(F)$ grows at most linearly with respect to
$\ddim(F)$ (see Theorem~\ref{main}).

Throughout, we let $p$ be a prime number and use the notation
$K^M_i(F)$ to denote the $i$'th Milnor K-group of $F$, and $k_i(F) =
K^M_i(F)/pK^M_i(F)$. By the Bloch-Kato conjecture, we may identify
$k_i(F) = H^i(F, \mu_p^{\otimes i})$. We will use the notation $(a_1,
\ldots, a_n)$ to denote a symbol in the group $k_n(F)$. In the case
that $\rho \in F$ is a primitive $p$'th root of unity, we will use the
notation $D_{(a, b)}$ to denote the symbol algebra generated by
elements $x, y$ and satisfying $x^p = a, y^p = b, xy = \rho yx$.

\subsection*{Acknowledgments}

The authors would like to thank Skip Garibaldi for useful
conversations, and help with finding references during the writing of
this manuscript.

\section{Homogeneous forms and Milnor K-theory}

\begin{defn}
By a form of degree $d$ and dimension $n$, we mean a homogeneous
degree $d$ polynomial function $f \in F[V^*]$ where $V$ is an
$n$-dimensional vector space. We will also refer to $f$ as a degree
$d$ form on $V$. We say that $f$ is isotropic if there is a $v \in
V\setminus\{0\}$ such that $f(v) = 0$.
\end{defn}

\begin{defn}
Suppose that $\alpha \in k_n(F)$, and let $V$ be a vector space. We
say that a form $N$ on $V$ 
\begin{itemize} 
\item
\textit{neutralizes $\alpha$} if $\alpha
\cup (N(v))$ is trivial in $k_n(L)$ for every field extension $L/F$
and every choice of $v \in V_L$ with $N(v) \neq 0$, 
\item
\textit{splits $\alpha$} if for every field extension $L/F$, we have
$N_L$ isotropic implies $\alpha_L = 0$.
\end{itemize} 
\end{defn}

\begin{defn}
Let $A$ be a unital finite dimensional strictly power associative
algebra, and let $N_A \in F[A^*]$ be the reduced norm form (see
Section~\ref{appendix}).  As in Section~\ref{appendix}, we say that
$A$ is principally division, if for every $a \in A$, the subalgebra
$F[a]$ generated by $a$ is a division algebra. We say that $A$ is
\textit{adapted to a class} $\alpha \in k_n(F)$ if $N_A$ is degree
$p$ and if for every field extension $L/F$, $A_L$ is principally
division if and only if $\alpha_L \neq 0$.
\end{defn}

\begin{lem}
Suppose that $A$ is a unital finite dimensional strictly power
associative algebra which is adapted to $\alpha \in k_n(F)$. Then
$N_A$ neutralizes and splits $\alpha$.
\end{lem}
\begin{proof}
Let $d = N_A(x)$ for $x \in A$. We claim that $\alpha \cup (d)
= 0$. If $x \in F \subset A$ then $N_A(x) = x^p = d$ and consequently
$\alpha \cup (d) = p \alpha \cup (x) = 0$. On the other hand, if $x
\in A \setminus F$ then $F(x)$ is a degree $p$ subfield of $A$ which
by assumption concerning the characteristic of $F$ is separable. After
passing to a prime-to-$p$ extension, which will not affect the
triviality of $\alpha \cup (d)$ due to the standard
restriction-corestriction argument, we may in fact assume that
$F(x)/F$ is a cyclic Kummer extension, say $E = F(x) = F(y)$ where $y \in
A$ with $y^d = b \in F$. 

Now, $A_E$ contains $E \otimes E$ which contains zerodivisors. Since
$E/F$ is separable, $E$ is principally generated over $F$, say $E =
F[a]$, and hence $E \otimes E = E[a \otimes 1]$ is also principally
generated.  It follows that $A_E$ is not principally division. Since
$N_A$ is adapted to $\alpha$, this implies in turn that $\alpha_E$ is
split. But now, by \cite[Theorem~5]{LMS:GMSGC} (see also
\cite[Proposition~5.2]{Voe:Z2}), we have the exact sequence:
\[\xymatrix{
	k_{n-1}(E) \ar[r]^{N_{E/F}} & k_{n-1}(F) \ar[r]^{\cup (b)} &
	k_{n}(F) \ar[r] & k_{n}(E)
}\]
which tells us that $\alpha = \beta \cup (b)$ for some $\beta \in
k_{n-1}(F)$. But now, we claim that $\alpha \cup (d) = 0$. In fact
this will come from the fact that $(b) \cup (d) = (b, d) = 0$. Again,
by the exact sequence above, this amounts to saying that $d \in
N_{E/F}(k_1(E)) = N_{E/F}(E^*)$. But $d = N_{E/F}(x), x \in E^*$ by
hypothesis.

Next, suppose that $N_A(x) = 0$, we need to show that $\alpha = 0$ in
this case. We will do this by showing that $A$ is not principally
division, which would imply $\alpha = 0$ by our hypothesis on $A$. 

Arguing by contradiction, suppose $A$ is principally division. Since
$N_A|_F$ is the $p$'th power operation on $F$, and $F$ is a field, it
follows that $x \in A \setminus F$ and hence $F(x) = E$ is a degree
$p$ separable field extension. Since $x$ satisfies its characteristic
polynomial $\chi_x(T)$ of degree $p$, it follows that $\chi_x(T)$ is
in fact the minimal polynomial of $x$ as an element of $E$. Since $E$
is a field, $\chi_x(T)$ must be irreducible. But, since $N_A(x) = 0$
is the constant coefficient of this polynomial, we find that
$\chi_x(T) = T f(T)$ for some monic polynomial $f(T)$ contradicting
the irreducibility of $\chi_x(T)$.
\end{proof}

\begin{ex}[Kummer extensions]
If $L = F(\sqrt[p]{a})$. Then $L$ is adapted to $(a)$. This follows
from the fact that $L/F$ is a division algebra exactly when $a \not\in
(F^*)^p$, which in turn happens exactly when $(a) \neq 0$ in
$k_1(F^*)$.
\end{ex}

\begin{ex}[Symbol algebras] \label{csa norm}
Assume $\mu_p \subset F$, and let $\alpha = (a, b) \in k_2(F) =
\Br(F)_p$. Let $D = D_{(a, b)}$ be the corresponding symbol algebra.
Then $D$ is adapted to $\alpha$. Indeed, since $D$ has degree $p$, it
is either a split algebra or a division algebra. Further, it is split
if and only if $(a, b) = 0 \in k_2(F) = \Br(F)_p$, and it is division
if and only if its norm form is anisotropic, and as we see in
Lemma~\ref{principal division}, it is division if and only if it is principally
division.
\end{ex}

\begin{ex}[First Tits process Albert algebras] \label{albert
norm}
Consider the case $p = $, suppose that $\mu_3 \in F$ and let $D =
D_{(a, b)}$ be the symbol algebra as before. Choose $c \in F^*$ and
let $A = D_0 \oplus D_1 \oplus D_2$ with $D_i = D$ be the Albert
algebra given by the first Tits process (see, for example
\cite[section~2.5]{PeRa}, or \cite[Section IX.39]{BofInv}).

We claim that $A$ is adapted to $(a, b, c)$. This follows from
\cite[Theorem~1.8 and Proposition~2.6]{PeRa}, where it is shown that
$A$ is division if and only if $(a, b, c) \neq 0$, and
Lemma~\ref{albert principal}, where we show that $A$ is principally
division if and only if it is division.
\end{ex}

We will have use of the following identity in $k_2(F)$
\begin{lem} \label{symbol identity}
Suppose $a, b \in F^*$ with $a + b \neq 0$. Then $(a, b) = (-ab^{-1},
a + b) = (a + b, -ba^{-1})$.
\end{lem}
\begin{proof}
Since passing to a prime-to-$p$ extension constitutes an injective map
on the level of $k_2$, we may assume that $\mu_p \subset F$, and in
particular, we may identify $k_2(F) = \Br(F)_p$ by the
Merkurjev-Suslin Theorem. In particular, it suffices to show that
$D_{(a, b)} \cong D_{(a + b, -ab^{-1})}$. We do this as follows:
presenting $D_{(a, b)}$ as generated by $x, y$ with $x^p = a, y^p = b,
yx = \rho xy$, we let $z = x + y$ and $w = -xy^{-1}$, we see $z^p =
x^p + y^p = a + b$,
$w^p=\N_{F[x]/F}(-x)b^{-1}=(-1)^p(-1)^{p-1}ab^{-1}=-ab^{-1}$, and 
\begin{multline*}
zw = -(x + y)xy^{-1} = -x^2y^{-1}t - yxy^{-1} = -x^2 y^{-1} - \rho
x \\ = -x(x y^{-1} + \rho y^{-1} y) = -x(\rho y^{-1} x + \rho
y^{-1}y) = -\rho xy^{-1}(x + y) = \rho w z
\end{multline*}
Hence $w, z$ generate the algebra $D_{(-ab^{-1}, a + b)}$ which
therefore must be isomorphic to $D_{(a, b)}$ as desired.
\end{proof}

\begin{prop} \label{norm doubling}
Suppose that $\alpha = \beta \cup (a)$, and that $N$ is a degree $p$
form defined on a vector space $V$ which neutralizes and splits
$\beta$. Then then the form $N' = N \oplus -aN$ on $V \oplus V$ splits
$\alpha$ and neutralizes $\alpha$. 
\end{prop}
\begin{proof}
Suppose that $v, w \in V$. We first show that $N'$ splits
$\alpha$. Suppose that we have $v, w \in V$ such that $N(v) -
aN(w) = 0$. We must show in this case that $\alpha = 0$. Note that we
may assume $N(v), N(w) \neq 0$ since otherwise, using the fact that
$N$ splits $\beta$, we would have $\beta = 0$ and hence
$\alpha = 0$ as well. But now we may write
\[\alpha = \beta \cup (a) = \beta \cup (N(w)) + \beta \cup (a) = \beta
\cup (aN(w)) = \beta \cup (N(v)) = 0\]

Now, to see that $N$ neutralizes $\alpha$, let $v, w \in V$. We
must show that $\alpha \cup (N(v) - aN(w)) = 0$. If either $N(v)$ or
$N(w)$ are zero, then $\beta = 0$ since $N$ splits
$\beta$, and hence $\alpha \cup (N(v) - aN(w)) = 0$ as well. Hence we
may assume that $N(w), N(v) \neq 0$. Then we have:
\begin{align*}
\alpha \cup (N(v) - aN(w)) &= \beta \cup (a) \cup (N(v) - aN(w)) \\
&= \beta \cup (N(w) N(v)^{-1} a) \cup (N(v) - aN(w)) \\
&= \beta \cup \big(-(-aN(w))N(v)^{-1}, N(v) - aN(w)\big)
\end{align*}
By Lemma~\ref{symbol identity}, we may then write
\begin{multline*}
\beta \cup
\big(-(-aN(w))N(v)^{-1}, N(v) - aN(w)\big)
 = \beta \cup (-aN(w), N(v)) \\=
\beta \cup (N(v)) \cup (-a^{-1} N(w)^{-1})
\end{multline*}
But $\beta \cup (N(v)) = 0$ since $N$ neutralizes $\beta$, hence
$\alpha \cup (N(v) - aN(w)) = 0$ as desired.
\end{proof}

Of course, in the case $p = 2$, the Pfister forms give examples of
forms which neutralize and split symbols. The contents of the
following Proposition were noted by Serre (see \cite[page
99]{Serre:CG}):
\begin{prop} \label{pfister splitting}
Suppose that $p = 2$ and $\alpha = (a_1, \ldots, a_n) \in k_n(F)$ is a
symbol. Let $q = \lbb a_1, \ldots, a_n\rbb$ be the corresponding
Pfister form. Then $q$ splits and neutralizes $\alpha$.
\end{prop}
\begin{proof}
By the Milnor conjectures, we have an isomorphism $e_n :
I^n(F)/I^{n+1}(F) \to k_n(F)$ sending $q$ to $\alpha$. If $q$ is
isotropic, then since it is a Pfister form, it must be split and in
particular represent the $0$ class in $I^n(F)/I^{n+1}(F)$. It then
follows that $\alpha = e_n(q)$ is zero as well.

To see that $q$ neutralizes $\alpha$, suppose that $a = q(v)$, and
consider $\alpha \cup (a) = e_{n+1}(q \otimes \left<1, -a\right>)$. The form $q
\otimes \left<1, -a\right> = q \perp -q(v) q$ is isotropic. This is because $q$,
being a Pfister form, represents $1$, say $q(w) = 1$, and then we may
write
\[(q \perp -q(v) q) (v, w) = q(v) - q(v)q(w)  = 0.\]
But as before, it follows that since $q \otimes \left<1, -a\right>$ is a Pfister
form, it must in fact be hyperbolic implying that $\alpha \cup (a) =
e_{n+1}(q \otimes \left<1, -a\right>) = 0$ as desired.
\end{proof}

Let us now record some applications of these results. The first few of
which are well known consequences of previous results in the area.

\begin{thm} \label{milnor serre}
Suppose that $F$ is $C_n$, and let $p = 2$. Then every element of
$k_n(F) = K_n^M(F)/2$ is a symbol and $k_{n+1}(F) = 0$. In particular,
$\cdim_2(F) \leq n$.
\end{thm}
\begin{proof}
For the first claim, suppose that $\alpha = (a_1, \ldots, a_n)$ and
$\beta = (b_1, \ldots, b_n)$ are symbols in $k_n(F)$. It suffices to
show that we may rewrite $\alpha$ and $\beta$ so that $a_i = b_i$ for
$i = 1, \ldots, n-1$. We do this by induction on the number of slots
that a presentation of $\alpha$ and $\beta$ share. Suppose that a
given presentation of $\alpha$ and $\beta$ share $m < n-1$ slots.
Without loss of generality, we may suppose that we have $a_i = b_i$
for $i = 1, \ldots, m$ for $m < n-1$. By Proposition~\ref{pfister
splitting}, we may find forms $F, G$ of degree $2$ on
$2^{n-1}$-dimensional vector spaces $V, W$ which split and neutralize
$(a_1, \ldots, a_{n-1})$ and $(b_1, \ldots, b_{n-1})$ respectively. Let
$\phi$ be the one dimensional form $t^2 a_{n-1} \in F[t]$. By the
$C_n$ property, the $2^n + 1$-dimensional form $\phi \oplus F \oplus
G$ has a nontrivial zero, say $a_{n-1} t_0^2 + F(v)a_n - G(w)b_n = 0$
for some $t_0 \in F$, $v \in V$ and $w \in W$.  If either $F(v)$ or
$G(w)$ is zero, then one of the forms $\alpha$ or $\beta$ is trivial
(since these forms split $\alpha$ and $\beta$), and we are done. Hence
we may assume that $F(v), G(w) \neq 0$. But now, we have
\begin{equation} \label{first}
(a_1, \ldots, a_n) = (a_1, \ldots, F(v) a_n) = (a_1, \ldots, t_0^2
a_{n-1}, F(v) a_n)
\end{equation}
If $t_0^2 a_{n-1} + F(v) a_n = 0$, it would follow by the definition
of Milnor K-theory that $(t_0^2 a_{n-1}, F(v) a_n) = 0$ and hence
$\alpha = 0$. Hence we may assume that this is not the case. By
Lemma~\ref{symbol identity}, we may therefore write 
\begin{multline*}
(t_0^2 a_{n-1}, F(v) a_n) = (-t_0^2 a_{n-1} (F(v) a_n)^{-1}, t_0^2
a_{n-1} + F(v) a_n) \\ = (-t_0^2 a_{n-1} (F(v) a_n)^{-1}, G(w) b_n)
\end{multline*}
Combining this with equation~(\ref{first}), it follows that we may
write
\[\alpha = (a_1, \ldots, a_n) = (a_1, \ldots, -t_0^2 a_{n-1} (F(v) a_n)^{-1},
G(w) b_n)\]
and writing $\beta = (b_1, \ldots, b_n) = (b_1, \ldots, b_{n-1}, G(w)
b_n)$, we see that our two new presentations now have $m+1$
slots in common. The claim follows.

For the second assertion, it suffices to show that every symbol
$\alpha = (a_1, \ldots, a_{n+1})$ is trivial in $k_{n+1}(F)$. But via
the Milnor conjectures, we may identify $\alpha$ as the $e_{n+1}$
invariant of the $(n+1)$-fold Pfister form $\lbb a_1, \ldots, a_{n+1}
\rbb$, which has dimension $2^{n+1}$. Since $F$ is $C_n$, it follows
that this Pfister form is isotropic and hence hyperbolic. Consequently
$\alpha$, its $e_{n+1}$-invariant, must vanish.
\end{proof}

We next recall a consequence of the work of Merkurjev and Suslin:
\begin{thm} \label{merksus}
Let $F$ be $C_2$ then $k_3(F) = K_3^M/p = 0$, and so $\cdim(F) \leq 2$.
\end{thm}
\begin{proof}
Let $\alpha = (a, b, c) \in k_3(F)$. Let $D = D_{(a, b)}$ be the
symbol algebra associated to $(a, b)$, and let $N_D$ be its reduced
norm.  By Example~\ref{csa norm}, $N_D$ splits and neutralizes $(a,
b)$. Therefore, by Proposition~\ref{norm doubling}, the form $N = N_D
\oplus -c N_D$ splits and neutralizes $\alpha$. But note that this
form has $2p^2 > p^2$ and degree $p$. Since $F$ is $C_2$, it follows
that it is isotropic and therefore $\alpha = 0$ as desired.
\end{proof}

\begin{thm} \label{albert 1}
Let $F$ be $C_3$. Then $K_4^M/3 = 0$, and so $\cdim_3(F) \leq 3$.
\end{thm}
\begin{proof}
If $A$ is the Albert algebra for $(a, b, c)$ then $N_A \oplus -d N_A$
is a $2 \cdot 3^3 = 54$ dimensional form splitting and neutralizing
$(a, b, c, d)$. since $54 > 27$ we are done.
\end{proof}

\begin{thm} \label{albert 2}
Let $F$ be $C_4$. Then $K_5^M/3 = 0$ and so $\cdim_4(F) \leq 4$.
\end{thm}
\begin{proof}
If $A$ is the Albert algebra for $(a, b, c)$ then $N = N_A \oplus -d N_A$
is a $2 \cdot 3^3 = 54$ dimensional form splitting and neutralizing
$(a, b, c, d)$, and $N' = N \oplus -e N$ is a $108$ dimensional form
splitting and neutralizing $(a, b, c, d, e)$. Since $108 > 81 = 3^4$ we are done.
\end{proof}

\begin{thm} \label{main}
Let $F$ be $C_n$. Then
\begin{enumerate}
\item $\cdim_2(F) \leq n$,
\item $\cdim_3(F) \leq \left\{
		\begin{matrix}
			n & \text if n \leq 4 \\
			\lceil(n-3)(\log_2(3)) + 3\rceil & \text{otherwise,} \\
		\end{matrix}\right.$
\item $\cdim_p(F) \leq \left\{
		\begin{matrix}
			n & \text if n \leq 2 \\
			\lceil(n-2)(\log_2(p)) + 1\rceil & \text{otherwise.} \\
		\end{matrix}\right.$
\end{enumerate}
\end{thm}
For example, if $F$ is $C_3$ then $\cdim_5(F), \cdim_7(F) \leq 4$. Note that
in particular, the $p$-cohomological dimension is bounded linearly
with respect to the Diophantine dimension, with the slope $\log_2(p)$.
\begin{proof}
The first part, corresponding to $\cdim_2$ follows from
Theorem~\ref{milnor serre}. 

For $\cdim_3$, the cases with $n \leq 4$ follow from
Propositions~\ref{albert 1} and \ref{albert 2} respectively. In
general, to check whether or not $\cdim_3(F) \leq m$, write a degree
$m+1$ symbol in $K_{m+1}^M/3$ as $\alpha = (a, b, c) \cup (d_1,
\ldots, d_{m-2})$. We wish to find a criterion on $m$ which will
guarantee that $\alpha = 0$. This works as follows: if $N_A$ is the
norm form for the first Tits process Albert algebra $A$ defined by
$(a, b, c)$, then inductively applying Proposition~\ref{norm doubling}
we obtain a form $N$ which splits $\alpha$ in $27 \cdot 2^{m-2}$
variables of degree $3$. In particular, by the $C_n$ property, we find
that this form is isotropic, and hence $\alpha$ is split when $27
\cdot 2^{m-2} > 3^n$ or $2^{m-2} > 3^{n - 3}$.  But this translates to
$m > (n-3)(\log_2(3)) + 3$, since $\log_2 3$ is irrational, this is
equivalent to saying $m \geq \lceil (n-3(\log_2(3)) + 3\rceil$.

The general case of $\cdim_p$ follows much like the case of $\cdim_3$
above, considering a symbol $\alpha = (a, b) \cup (c_1, \ldots,
c_{m-1})$, starting with Theorem~\ref{merksus}  and applying
Proposition~\ref{norm doubling} to obtain a form $N$ of degree $p$ and
dimension $p^2 \cdot 2^{m-1}$ splitting $\alpha$. Again, in this case,
the $C_n$ property guarantees that $\alpha$ is split in case $2^{m-1}
> p^{n - 2}$. The result follows.
\end{proof}

\begin{rem}
Although it is a weaker result, it is interesting to note that one may
associate an ``obvious form'' which splits a given symbol. In
particular if $\alpha = (a_1, \ldots, a_n) \in k_n(F)$, then the form
\(a_1 t_1^p + \cdots + a_n t_n^p\) splits $\alpha$.
\end{rem}
\begin{proof}
We induct on $n$, the case $n = 1$ being trivial and the case $n = 2$
following from the fact that $N_{D_{(a, b)}}$ when restricted to
the subspace $Fx + Fy$ is exactly the form $at^p + bs^p$. For the
general induction step, suppose that $\sum a_i r_i^p = 0$ for some
choice of $r_i \in F$, and set $u = \sum_{i = 0}^{n-2} a_i r_i^p$.
Then $-u = a_{n-1} r_{n-1}^p + a_n r_n^p$. By Lemma~\ref{symbol
identity}, we may write $(a_{n-1}, b_{n-1}) = (-u, v)$ for some $v \in
F^*$. But therefore we may write $\alpha = (a_1, \ldots, a_{n-2}, -u,
v)$. Considering the form $(a_1, \ldots, a_{n-2}, -u)$ we find that by
$\sum_{i = 1}^{n-2} a_i r_i^p + (-u) 1 = 0$, the induction hypothesis
implies $(a_1, \ldots, a_{n-2}, -u) = 0$ and so $\alpha = (a_1,
\ldots, a_{n-2}, -u) \cup (v) = 0$ as well.
\end{proof}

\section{Appendix: Norm forms on power associative algebras}
\label{appendix}

Recall that an algebra $A$ is called \textit{power associative} if all
associators of the form $\{a^i, a^j, a^k\}$ vanish, or equivalently,
if for every $a \in A$, the subalgebra $F[a]$ generated by $a$ is
associative and commutative. We say that $A$ is \textit{strictly
power associative} if for every field extension $L/F$, the algebra
$A_L = A \otimes_F L$ is power associative. Note that in case $F$ has
characteristic not equal to $2$, by linearizing associator
relations such as $\{x, x, y\}$, one sees that every power
associative algebra is automatically strictly power associative.

Let $A$ be a unital finite dimensional strictly power associative $F$-algebra.
We now recall the definition of the reduced norm function on $A$.  Set
$R = F[A^*]$ to be the ring of polynomial functions on $A$.  Consider
the generic element $x \in A_R$, defined as follows: if $a_1, \ldots
a_n$ is a basis for $A$ and $f_1, \ldots, f_n \in A^* \subset R =
F[f_1, \ldots, f_n]$ is the corresponding dual basis, then $x = \sum
a_i f_i$. In a coordinate free way, we may also write $x$ as the
identity map, thought of as an element of $Hom_F(A, A) = A \otimes A^*
\subset A \otimes R = A_R$. We may now consider the subalgebra $R[x]
\subset A_R$. Since $A_R$ is a finitely generated module and $R$ is
Noetherian, it follows that $R[x]$ is finitely generated $R$-module
and hence an integral extension. Let $m(T) = m_0 + m_1 T + \cdots +
m_{r-1} T^{r-1} + T^r \in R[T]$ be the minimal polynomial of $T$.
Since $m(x)$ is the zero polynomial function on $A$, it follows that
for any $a \in A$, $a$ is a root of function 
\[ \chi_a(T) = m_0(a) + m_1(a)T + \cdots + m_{r-1}(a) T^{r-1} + T^r\]
which we call the reduced characteristic polynomial of $a$. We define
the reduced norm of $a$, $N_A(a)$ to be $m_0(a)$. The reduced norm
function is the function $N_A = m_0 \in R$. We define the adjunct of
$a$, defined by $\cof_A(a)$ to be $m_1(a) + m_2(a) a + \cdots
m_{r-1}(a) a^{r-2} + a^{r-1}$. We then have $a (\cof_A(a)) = N_A(a)$.
Note that since these coincide with the standard characteristic and
minimal polynomials of the $R$-algebra $R[x]$ acting on itself, we
have the familiar identity $\chi_a(T) = N_A(T - a)$.

We say that a power associative finite dimension $F$-algebra $A$ has
degree $d$ if $\chi_A$ (and hence also $N_A$) is a degree
$d$ polynomial function.

Recall that a nonassociative algebra $A$ is called division in case
left and right multiplication by every nonzero element $a \in A$ are
both bijective. We say that $A$ is \textit{principally division}, if
for every $a \in A$, the subalgebra $F[a]$ generated by $a$ is
division. We say that $a \in A$ is invertible if there exists $b$
such that $ab = 1 = ba$.
\begin{lem} \label{pow prin div}
Suppose that $A$ is a power associative $F$-algebra of prime degree
$p$. Then $A$ is principally division if and only if $N_A$ is
anisotropic.
\end{lem}
\begin{proof}
Suppose $N_A$ is anisotropic and let $a \in A$. We must show that
$F[a]$ is division. Let $b \in F[a] \subset A$ with $b \neq 0$. Since
$N_A(b) \neq 0$, we may write $b (N_A(b))^{-1} (\cof_A(b)) = 1$, and
since $\cof_A(b) \in F[b] \subset F[a]$, it follows that $b$ is
invertible in $F[a]$. Therefore $F[a]$ is division.

Conversely, if $A$ is principally division, we wish to show that for
all $a \in A\setminus 0$, $N_A(a) \neq 0$. Let $m_a(t)$ be the
minimal polynomial of $a$. It follows from
\cite{Jac:GTJ}, that $m_a(t)$ and $\chi_a(t)$ have the same prime
factors. Since $F[a] = F[t]/(m_a(t))$ is a division algebra, it
follows that $m_a(t)$ is irreducible, and in particular, not
divisible by $t$. But therefore $t$ cannot divide $\chi_a(t)$ as
well. It therefore follows that $N_A(a)$, being (up to a sign) the
constant coefficient of $\chi_a(t)$ is also nonzero, as desired.
\end{proof}

\begin{lem} \label{principal division}
Suppose $A$ is a finite dimensional alternative algebra. Then $A$ is
principally division if and only if it is division.
\end{lem}
\begin{proof}
Certainly if $A$ is division, it is also a domain. Since each
subalgebra $F[a]$ is then a finite dimensional domain, they are
domains, and hence $A$ is principally division.

For the converse, suppose that $A$ is principally division, and let
$a \in A$. We wish to show that left multiplication by $a$ gives an
bijection from $A$ to itself. The proof that right multiplication
also has this property follows from the same argument.

To see this, we start from the fact
that $F[a]$ is a field, and, using the fact that $A$ is left
alternative, left multiplication gives $A$ the
structure of a vector space over $F[a]$. But scalar multiplication by
a nonzero scalar is automatically an isomorphism of $A$ with inverse
given by $a^{-1}$. 
\end{proof}

\begin{lem} \label{albert principal}
Let $A$ be a first Tits process Albert algebra over $F$. Then $A$ is division,
if and only if it is principally division.
\end{lem}
\begin{proof}
Clearly if $A$ is division it is principally division.  For the
converse, suppose that $A$ is not division.  By \cite[Chapter~9,
Theorem~20]{Jac:SRJA}, $A$ is split, and hence must contain the
algebra $M_n(F)^+$. But it follows that $A$ is therefore not
principally division, since, for example, it contains nilpotent
elements.
\end{proof}
 
\bibliographystyle{alpha}
\def\cprime{$'$} \def\cprime{$'$} \def\cprime{$'$} \def\cprime{$'$}
  \def\cprime{$'$} \def\cftil#1{\ifmmode\setbox7\hbox{$\accent"5E#1$}\else
  \setbox7\hbox{\accent"5E#1}\penalty 10000\relax\fi\raise 1\ht7
  \hbox{\lower1.15ex\hbox to 1\wd7{\hss\accent"7E\hss}}\penalty 10000
  \hskip-1\wd7\penalty 10000\box7}

\bibliography{citations}

\end{document}